\documentclass{amsart}
\usepackage{amsfonts,amssymb,amsmath,amsthm,mathrsfs}
\usepackage{url}
\usepackage{enumerate}

\usepackage[pdftex, pagebackref, bookmarksnumbered, bookmarksopen, colorlinks, citecolor=blue, linkcolor=blue]{hyperref}

\newtheorem{theorem}{Theorem}[section]
\newtheorem{lemma}[theorem]{Lemma}

\newtheorem{corollary}[theorem]{Corollary}
\theoremstyle{definition}

\numberwithin{equation}{section}

\author[X. Tao]{Xiangxing Tao}
\address{Xiangxing Tao: Department of Mathematics, School of Science, Zhejiang University of Science and Technology,
Hangzhou 310023, People's Republic of China}
\email{xxtao@zust.edu.cn}

\author[G. Hu]{Guoen Hu}
\address{Guoen Hu:  Department of Mathematics, School of Science, Zhejiang University of Science and Technology,
Hangzhou 310023, People's Republic of China}
\email{guoenxx@163.com}
\thanks{The research of the first author was supported by the NNSF of
China under grant \#12071437}
\thanks{Guoen Hu is the corresponding author}
\begin{document}
\title[rough operator]{A bilinear sparse domination for the  maximal  singular integral operators with rough kernels}

\begin{abstract}
Let $\Omega$ be homogeneous of degree zero, integrable on $S^{d-1}$ and   have mean value zero, $T_{\Omega}$  be the homogeneous singular integral operator  with kernel $\frac{\Omega(x)}{|x|^d}$ and $T_{\Omega}^*$ be the maximal operator associated to $T_{\Omega}$. In this paper, the authors prove  that if $\Omega\in L^{\infty}(S^{d-1})$, then for all $r\in (1,\,\infty)$, $T_{\Omega}^*$ enjoys a $(L^\Phi,\,L^r)$ bilinear sparse domination with $\Phi(t)=t\log\log ({\rm e}^2+t)$. Some applications of this bilinear sparse domination are also given.
\end{abstract}
\maketitle

\section{Introduction}
In this paper, we will work on $\mathbb{R}^d$, $d\geq 2$. Let $\Omega$ be
homogeneous of degree zero, integrable on
the unit sphere ${S}^{d-1}$ and have mean value zero. Define the  singular integral operator
${T}_{\Omega}$ by
\begin{eqnarray}\label{eq1.1}{T}_{\Omega}f(x)={\rm p.\,v.}\int_{\mathbb{R}^d} \frac
{\Omega( y')}{|y|^d}f(x-y)dy,\end{eqnarray}  where and in the following, $y'=y/|y|$ for $y\in\mathbb{R}^d$. This operator was introduced by
Calder\'on and Zygmund \cite{cz1}, and then studied by many authors
in the last sixty years.  Calder\'on and Zygmund \cite {cz2} proved that
if $\Omega\in L\log L({S}^{d-1})$, then $T_{\Omega}$ is bounded on
$L^p(\mathbb{R}^d)$ for $p\in (1,\,\infty)$.   Ricci
and Weiss \cite{rw} improved the result of Calder\'on and Zygmund, and
showed that $\Omega\in H^1(S^{d-1})$ guarantees the
$L^p(\mathbb{R}^d)$ boundedness on $L^p(\mathbb{R}^d)$ for $p\in
(1,\,\infty)$. Seeger \cite{se} showed that $\Omega\in L\log
L(S^{d-1})$ is a sufficient condition such that $T_{\Omega}$ is bounded
from $L^1(\mathbb{R}^d)$ to $L^{1,\,\infty}(\mathbb{R}^d)$, see also \cite{chr2}.

Let   $A_p(\mathbb{R}^d)$ $(p\in [1,\,\infty))$ be the weight function class of Muckenhoupt, that is, $$A_{p}(\mathbb{R}^d)=\{w: \, w\,\hbox{is\,\,nonnegative\,\,and\,\,locally\,\,integrable\,\,in}\,\,\mathbb{R}^d,\,[w]_{A_p}<\infty\},$$
where and in what follows, $$[w]_{A_p}:=\sup_{Q}\Big(\frac{1}{|Q|}\int_Qw(x)dx\Big)\Big(\frac{1}{|Q|}\int_{Q}w^{1-p'}(x)dx\Big)^{p-1},\,\,\,p\in (1,\,\infty),$$
and
$$[w]_{A_1}:=\sup_{x\in\mathbb{R}^d}\frac{Mw(x)}{w(x)},$$
with  $M$   the Hardy-Littlewood maximal operator.  Given $w\in A_p(\mathbb{R}^d)$, $1\leq p<\infty$, the constant $[w]_{A_p}$ just defined is known the $A_p$ constant of $w$ (see  \cite[Chapter 9]{graf1} for the properties of $A_p(\mathbb{R}^d)$). For $w\in \cup_{p>1}A_p(\mathbb{R}^d)$,
$[w]_{A_{\infty}}$,  the $A_{\infty}$ constant of $w$,  is defined by
$$[w]_{A_{\infty}}=\sup_{Q\subset \mathbb{R}^d}\frac{1}{w(Q)}\int_{Q}M(w\chi_Q)(x)dx,$$
see \cite{fuj} or \cite{wil}. Duoandikoetxea and Rubio de Francia \cite{drf} considered the   boundedness on weighted $L^p$ spaces with $A_p$ weights  for $T_{\Omega}$ when $\Omega\in L^q(S^{d-1})$ for some $q\in (1,\,\infty)$, see also
\cite{wat}.
In the last
decade, considerable attention has been paid to the quantitative weighted bounds for the operator $T_{\Omega}$. The first work in this area is the paper of  Hyt\"onen, Roncal  and Tapiola
\cite{hyta}.
By a quantitative weighted estimate for the Calder\'on-Zygmund operators,  approximation to the identity and interpolation, Hyt\"onen et al.
(see Theorem 1.4 in \cite{hyta}) proved that
\begin{theorem}\label{dingli1.1} Let $\Omega$ be homogeneous of degree zero,   have mean value zero on $S^{d-1}$ and $\Omega\in L^{\infty}(S^{d-1})$.
Then for $p\in (1,\,\infty)$ and $w\in A_{p}(\mathbb{R}^d)$,
\begin{eqnarray}\label{eq:1.estimate}\|T_{\Omega}f\|_{L^p(\mathbb{R}^d,\,w)}&\lesssim & \|\Omega\|_{L^{\infty}(S^{d-1})}[w]_{A_p}^{\frac{1}{p}}\max\{[w]_{A_{\infty}}^{\frac{1}{p'}},\,[w^{1-p'}]_{A_{\infty}}^{\frac{1}{p}}\}\\
&&\times\max\{[w]_{A_{\infty}},\,[w^{1-p'}]_{A_{\infty}}\}
\|f\|_{L^p(\mathbb{R}^d,\,w)}.\nonumber
\end{eqnarray}
\end{theorem}
For a family of cubes $\mathcal{S}$, we say that $\mathcal{S}$ is $\eta$-sparse, $\eta\in (0,\,1)$, if  for each fixed $Q\in \mathcal{S}$, there exists a measurable subset $E_Q\subset Q$, such that $|E_Q|\geq \eta|Q|$ and $\{E_{Q}\}$ are pairwise disjoint.

Let $\Phi_1,\,\Phi_2$ be two Young functions, and $\mathcal{S}$ be a sparse family. Associated with $\Phi_1,\,\Phi_2$  and $\mathcal{S}$, the bilinear sparse operator $\mathcal{A}_{\mathcal{S},\,\Phi_1,\,\Phi_2}$ is defined  by
$$\mathcal{A}_{\mathcal{S},\,\Phi_1,\,\Phi_2}(f_1,\,f_2)=\sum_{Q\in\mathcal{S}}\langle |f_1|\rangle_{\Phi_1,Q}\langle |f_2|\rangle_{\Phi_2,Q}|Q|.
$$
where and in what follows, $\langle |f|\rangle_{\Phi_1,\,Q}$ is the Luxemburgh norm defined by
$$\langle |f|\rangle_{\Phi_1,Q}=\inf\Big\{\lambda>0:\,\frac{1}{|Q|}\int_{Q}\Phi_1\big(\frac{|f(x)|}{\lambda}\big)dx\leq 1\Big\},$$
see \cite{rr}. For the case of $\Phi_1(t)=t^p$, we denote  $\langle f\rangle_{\Phi_1,\,Q}$ by $\langle f\rangle_{p,Q}$  for simplicity.

Conde-Alonso,   Culiuc,  Di Plinio and  Ou   \cite{ccdo}  proved that for all $p\in (1,\,\infty)$,  bounded functions $f$ and $g$ with compact supports,
$$
\Big|\int_{\mathbb{R}^d}g(x)T_{\Omega}f(x)dx\Big|\lesssim p'\|\Omega\|_{L^{\infty}(S^{d-1})}\sup_{\mathcal{S}}\mathcal{A}_{\mathcal{S},\,L^1,L^p}(f,\,g),$$
where the supremum is take over sparse families.
Using this bilinear sparse domination  and  some new estimates for bilinear sparse operators, Li, P\'erez, Rivera-Rios and  Roncal \cite{lpr} proved that
\begin{eqnarray*}\|T_{\Omega}f\|_{L^p(\mathbb{R}^d,\,w)}&\lesssim & \|\Omega\|_{L^{\infty}(S^{d-1})}[w]_{A_p}^{\frac{1}{p}}\max\{[w]_{A_{\infty}}^{\frac{1}{p'}},\,[w^{1-p'}]_{A_{\infty}}^{\frac{1}{p}}\}\\
&&\times\min\{[w]_{A_{\infty}},\,[w^{1-p'}]_{A_{\infty}}\}
\|f\|_{L^p(\mathbb{R}^d,\,w)},\nonumber
\end{eqnarray*}
which improved (\ref{eq:1.estimate}). Moreover, Li et al. \cite{lpr} proved that for any $w\in A_1(\mathbb{R}^d)$,
\begin{eqnarray*}\|T_{\Omega}f\|_{L^{1,\,\infty}(\mathbb{R}^d,\,w)}&\lesssim & \|\Omega\|_{L^{\infty}(S^{d-1})}[w]_{A_1}[w]_{A_{\infty}}\log ({\rm e}+[w]_{A_{\infty}})
\|f\|_{L^1(\mathbb{R}^d,\,w)},
\end{eqnarray*}
and established the weighted inequality of Coifman-Fefferman type that for $p\in [1,\,\infty)$ and $w\in A_{\infty}(\mathbb{R}^d)$, $$\|T_{\Omega}f\|_{L^p(\mathbb{R}^d,\,w)}\lesssim\|\Omega\|_{L^{\infty}(S^{d-1})}
[w]_{A_{\infty}}^2\|Mf\|_{L^p(\mathbb{R}^d,\,w)}.$$

Now let $T_{\Omega}^*$ be the maximal singular integral operator associated to $T_{\Omega}$, that is,
$$T_{\Omega}^*f(x)=\sup_{\epsilon>0}\Big|\int_{|y|>\epsilon}\frac{\Omega(y)}{|y|^d}f(x-y)dy\Big|.
$$
Also, $\Omega\in L\log L(S^{d-1})$ is a sufficient condition such that $T_{\Omega}^*$ is bounded on $L^p(\mathbb{R}^d)$ for all $p\in (1,\,\infty)$ (see \cite{cz2}). Duoandikoetxea and Rudio de Francia \cite{drf} proved that if $\Omega\in L^{\infty}(S^{d-1})$, then $T_{\Omega}^*$ is bounded on $L^p(\mathbb{R}^d,\,w)$ for all $p\in (1,\,\infty)$ and $w\in A_p(\mathbb{R}^d)$. For a long time, the weak type endpoint estimate for $T_{\Omega}^*$ is an open problem even for $\Omega\in L^{\infty}(S^{d-1})$. Honz\'{i}k \cite{hon} established a local weak type endpoint estimate for $T_{\Omega}^*$ when $\Omega\in L^{\infty}(S^{d-1})$.  Bhojak and Mohanty \cite{bhmo} improved the result of Honz\'{i}k, and proved the following result.
\begin{theorem}\label{dingliweak}
Let $\Omega$ be homogeneous of degree zero, have mean value zero on $S^{d-1}$, and $\Omega\in L^{\infty}(S^{d-1})$. Then for all $\alpha>0$,
\begin{eqnarray}\label{eq1.weak}
&&w\big(\{x\in\mathbb{R}^d:\,T_{\Omega}^*f(x)>\alpha\}\big)\\
&&\quad\lesssim [w]_{A_1}[w]_{A_{\infty}}\log ({\rm e}+[w]_{A_{\infty}})
\int_{\mathbb{R}^d}\Phi\big(\frac{|f(x)|}{\alpha}\big)dx.\nonumber
\end{eqnarray}
Throughout this paper,  $\Phi(t)=t\log\log ({\rm e}^2+t)$.
\end{theorem}

Di Plinio, Hyt\"onen and Li \cite{dhl} established a bilinear sparse domination for $T_{\Omega}^*$ when $\Omega\in L^{\infty}(S^{d-1})$. They proved that
\begin{theorem}\label{dingli1.3origin}
Let $\Omega$ be homogeneous of degree zero, have mean value zero on $S^{d-1}$, and $\Omega\in L^{\infty}(S^{d-1})$.   Then for each $p\in (1,\,\infty)$, bounded functions $f$ and $g$ with compact supports, there exists a sparse family $\mathcal{S}$, such that  $$\Big|\int_{\mathbb{R}^d}g(x)T_{\Omega}f(x)dx\Big|\lesssim p'\|\Omega\|_{L^{\infty}(S^{d-1})}\sup_{\mathcal{S}}\mathcal{A}_{\mathcal{S},\,L^p,L^p}(f,\,g).$$
\end{theorem}
Moreover, Di Plinio et al. \cite{dhl} established the following weighted norm inequality of Coifman-Fefferman type from their sparse estimate (see \cite[Theorem B]{dhl}).

\begin{theorem}\label{dinglidhl} Let $\Omega$ be homogeneous of degree zero,   have mean value zero on $S^{d-1}$ and $\Omega\in L^{\infty}(S^{d-1})$. Let $r\in (1,\,\infty)$.
Then for $p\in [1,\,\infty)$ and  $w\in A_{\infty}(\mathbb{R}^d)$,
$$\|T^*_{\Omega}f\|_{L^p(\mathbb{R}^d,\,w)}\lesssim r'\|\Omega\|_{L^{\infty}(S^{d-1})}[w]_{A_{\infty}}^2\|M_rf\|_{L^p(\mathbb{R}^d,\,w)}.$$
provided that $f$ is a bounded function with compact support, where and in the following, $M_{r}f(x)=[M(|f|^r)(x)]^{1/r}$.
\end{theorem}
The purpose of this paper is to establish a more refined bilinear sparse domination for $T_{\Omega}^*$. Our main result can be stated as follows.
\begin{theorem}\label{dingli1.3}
Let $\Omega$ be homogeneous of degree zero, have mean value zero on $S^{d-1}$, and $\Omega\in L^{\infty}(S^{d-1})$.   Then for each $r\in (1,\,\infty)$ and bounded function $f$, there exists a sparse family $\mathcal{S}$, such that for all bounded function $g$ with compact support,
\begin{eqnarray}\label{eq1.bilinear}&&\Big|\int_{\mathbb{R}^d}g(x)T^*_{\Omega}f(x)dx\Big|\lesssim \|\Omega\|_{L^{\infty}(S^{d-1})}\big(r'\mathcal{A}_{\mathcal{S},\,L^{1},L^r}(f,\,g)+
\mathcal{A}_{\mathcal{S},L^{\Phi},L^r}(f,g)\big).
\end{eqnarray}
\end{theorem}

Theorem \ref{dingli1.3} leads to an alternative proof of Theorem \ref{dingliweak}. We will show this in Section 3. Moreover,   the sparse domination (\ref{eq1.bilinear}) is more refined than the sparse domination in   Theorem \ref{dingli1.3origin}. This follows from the fact that for $r\in (1,\,\infty)$ and cube $Q\subset \mathbb{R}^d$,
\begin{eqnarray}\label{equa4.1} \langle|f |\rangle_{\Phi,\,Q}\lesssim \log r'\langle|f|\rangle_{Q}+\langle |f|\rangle_{r,\,Q}.\end{eqnarray}
To prove this, we may assume that $\langle |f|\rangle_{r,\,Q}=1$. Write
\begin{eqnarray*}
\int_{Q}\Phi(|f(x)|)dx&=&\int_{Q(|f|>4{\rm e}^2)}\Phi(|f(x)|)dx+\int_{Q(|f|\leq 4{\rm e}^2)}\Phi(|f(x)|)dx\\
&\leq &\int_{Q(|f|>4{\rm e}^2)}\Phi(|f(x)|)dx+2\int_{Q}|f(x)|dx.
\end{eqnarray*}
Recall that  when $t>4{\rm e}^2$, $\log t\leq \frac{t^{r-1}}{r-1}$ and
\begin{eqnarray*}\log \log ({\rm e}^2+t)\leq \log \Big(\frac{(2 t)^{r-1}}{r-1}\Big)\leq (r-1)\log 2+t^{r-1}+\log (\frac{1}{r-1}) . \end{eqnarray*}
This implies that
\begin{eqnarray*}
\int_{Q(|f|>4{\rm e}^2)}\Phi(|f(x)|)dx \lesssim \int_Q|f(x)|^{r}dx+\int_Q|f(x)|dx\log r'.
\end{eqnarray*}
Combining the estimate above leads to that
$$\frac{1}{|Q|}\int_{Q}\Phi(|f(x)|)dx \lesssim1+\log ( r')\langle|f|\rangle_Q,
$$
which yields (\ref{equa4.1}).

We make some conventions. In what follows, $C$ always denotes a
positive constant that is independent of the main parameters
involved but whose value may differ from line to line. We use the
symbol $A\lesssim B$ to denote that there exists a positive constant
$C$ such that $A\le CB$.   Constant with subscript such as $c_1$, does not change
in different occurrences. For a set $E\subset\mathbb{R}^d$,
$\chi_E$ denotes its characteristic function. For a cube $Q$, we use $\ell(Q)$ to denote the side length of $Q$, and $\lambda Q$  denotes the cube with the same center as $Q$  whose
side length is $\lambda$ times that of $Q$. For a suitable function $f$, we denote $\widehat{f}$ the Fourier transform of $f$.

\section{An endpoint estimate}
Given a sublinear operator $T$, define the maximal operator ${M}_{\lambda,\,T}$ by
$${M}_{\lambda,\,T}f(x)=\sup_{Q\ni x}\big(T(f\chi_{\mathbb{R}^d\backslash 3Q})\chi_{Q}\big)^*(\lambda |Q|),\,\,(0<\lambda<1),$$
where the supremum is taken over all cubes $Q\subset \mathbb{R}^d$ containing $x$, and $h^*$ denotes the non-increasing rearrangement of $h$.
The operator ${M}_{\lambda,\,T}$ was introduced by Lerner \cite{ler5}, and is useful in the study of bilinear sparse dominations for the operator $T_{\Omega}$, see \cite{ler5,hulai}. Also, for $p\in (0,\,\infty]$, define the maximal operator $\mathscr{M}_{p,T}$ by
$$\mathscr{M}_{p,\,T}f(x)=\sup_{Q\ni x}\Big(\frac{1}{|Q|}\int_{Q}|T(f\chi_{\mathbb{R}^d\backslash 3Q})(y)|^pdy\Big)^{1/p}.
$$
We denote $\mathscr{M}_{\infty,\,T}$ by $\mathcal{M}_{T}$.
Obviously, for any $x\in\mathbb{R}^d$ and  $p\in (0,\,\infty]$,
$$M_{\lambda,\,T}f(x)\leq \lambda^{-1/p}\mathscr{M}_{p,\,T}f(x).$$

Let $\eta\in C^{\infty}_0(\mathbb{R}^d)$ be a radial function  such that ${\rm supp}\, \eta\subset\{x:\frac{1}{2}\leq |x|\leq 2\}$ and
$\sum_{j\in\mathbb{Z}}\eta_j(x)=1$ for all $x\in \mathbb{R}^d\backslash\{0\}$, where $\eta_j(x)=\eta(2^{-j}x)$.
For $j\in\mathbb{Z}$ and  a function $\Omega$ on $S^{d-1}$, let $K_j(x)=\frac{\Omega(x)}{|x|^{d}}\eta_j(x)$. Then the singular integral operator $T_\Omega$ can be written as $T_\Omega f(x) = \sum_{j\in\mathbb{Z}}K_j*f(x)$. Let $T_{\Omega}^{**}$ be the lacunary maximal operator associated with $T_{\Omega}$, defined by
$$T_{\Omega}^{**}f(x)=\sup_{k\in\mathbb{Z}}\Big|\sum_{j\geq k}\int_{\mathbb{R}^d}K_j(x-y)f(y)dy\Big|.$$
Our main purpose in this section is to establish the following weak type endpoint estimate for $M_{\lambda,\,T_{\Omega}^{**}}$.

\begin{theorem}\label{dingli2.1}
Let $\Omega$ be homogeneous of degree zero and have mean value zero.
Suppose that $\Omega\in L^{\infty}(S^{d-1})$ with $\|\Omega\|_{L^{\infty}(S^{d-1})}=1$. Then for $\lambda\in (0,\,1)$,
\begin{eqnarray}\label{eqn2.8}
&&|\{x\in\mathbb{R}^d:M_{\lambda,T_{\Omega}^{**}}f(x)>\alpha\}|\lesssim \big(1+\log \big(\frac{1}{\lambda}\big)\big)\int_{\mathbb{R}^d}\frac{|f(x)|}{\alpha}dx+\int_{\mathbb{R}^d}\Phi\big(\frac{|f(x)|}{\alpha}\big)dx.\end{eqnarray}
\end{theorem}

To prove Theorem \ref{dingli2.1}, we will employ some lemmas.

\begin{lemma}\cite{ler5} \label{lem2.2}
Let $T$ be a sublinear operator which is bounded on $L^2(\mathbb{R}^d)$. Then for $\lambda\in (0,\,1)$,
$$\|M_{\lambda,\,T}f\|_{L^{2,\,\infty}(\mathbb{R}^d)}\lesssim \lambda^{-1/2}\|T\|_{L^2(\mathbb{R}^d)\rightarrow L^2(\mathbb{R}^d)}\|f\|_{L^2(\mathbb{R}^d)}.$$
\end{lemma}

We  recall  some definitions  from \cite{lernaza}. Given
a cube $Q_0$, we denote by $\mathcal{D}(Q_0)$ the   set of dyadic cubes with respect to $Q_0$, that is, the cubes from $\mathcal{D}(Q_0)$ are formed by repeating subdivision of $Q_0$ and each of descendants into $2^n$ congruent sub-cubes.

Let $\mathscr{D}$ be a  collection of cubes in $\mathbb{R}^d$. We say  that $\mathscr{D}$  is a dyadic lattice if
\begin{itemize}
\item[\rm (i)]if $Q\in\mathscr{D}$, then $\mathcal{D}(Q)\subset \mathscr{D}$;
\item[\rm (ii)] every two cubes $Q',\, Q''\in\mathscr{D}$ have a common ancestor,  that is, there
exists $Q\in\mathscr{D}$ such that $Q',\,Q'' \in \mathcal{D}(Q)$;
\item[\rm (iii)] for every compact set $K\subset \mathbb{R}^d$,  there exists a cube $Q\in\mathscr{D}$
containing $K$.
\end{itemize}

We also have the following three lattice theorem, see \cite[Theorem 2.4]{ler5} or \cite{lernaza}.
\begin{lemma}\label{lem2.3}  For every dyadic lattice
$\mathscr{D}$, there exist $3^d$ dyadic lattices $\mathscr{D}_1$,\,\dots,\,$\mathscr{D}_{3^d}$  such that
$$\{3Q : Q \in\mathscr{D}\} = \cup_{j=1}^{3^d}\mathscr{D}_j$$
and for every cube $Q \in \mathscr{D}$ and $j = 1,\,\dots,\, 3^d$, there exists a unique cube
$R \in  \mathscr{D}_j$ of side length $\ell(R)=3\ell(Q)$ containing $Q$.
\end{lemma}

Let  $\psi\in C^{\infty}_0(\mathbb{R}^d)$ be a nonnegative radial function such that
${\rm supp}\, \psi\subset \{x:\,|x|\leq 1/4\}$ and $\int_{\mathbb{R}^d}\psi(x)dx=1$.
For $j\in \mathbb{Z}$, set $\psi_j(x) = 2^{-jd}\psi(2^{-j}x)$. For a positive integer $l$, define
$$K^{l}(x)=\sum_{j\in\mathbb{Z}}K_j*\psi_{j-l}(x).$$
Let $T_l$, $W_l$,  $T_l^{**}$ be the operators defined by
$$T_lf(x)=K^l*f(x),$$
$$
W_lf(x)=\sum_{j\in\mathbb{Z}}\big(K_j-K_j*\psi_{j-l}\big)*f(x),
$$
and
$$
T_l^{**}f(x)=\sup_{k\in\mathbb{Z}}\Big|\sum_{j\geq k}\int_{\mathbb{R}^d}K_j*\psi_{j-l}(x-y)f(y)dy\Big|
$$
respectively.

\begin{lemma}\label{lem2.4}
Let $l\in\mathbb{N}$, $\Omega$ be homogeneous of degree zero and have mean value zero, and let $\Omega\in L^{\infty}(S^{d-1})$. Then, for every $0<\kappa<1$, one has
\begin{eqnarray}\label{eqn2.1}
\|T_{\Omega}f-T_{l}f\|_{L^2(\mathbb{R}^d)}\lesssim 2^{-\kappa l}\|\Omega\|_{L^{\infty}(S^{d-1})}\|f\|_{L^2(\mathbb{R}^d)}
\end{eqnarray}
and
\begin{eqnarray}\label{eqn2.2}
\|W^{**}_{l}f\|_{L^2(\mathbb{R}^d)}\lesssim 2^{-\kappa l}\|\Omega\|_{L^{\infty}(S^{d-1})}\|f\|_{L^2(\mathbb{R}^d)},
\end{eqnarray}
where and in the following,
\begin{eqnarray*}
W^{**}_{l}f(x)=\sup_{k\in\mathbb{Z}}\Big|\sum_{j\geq k}\int_{\mathbb{R}^d}\big(K_j*\psi_{j-l}(x-y)-K_j(x-y)\big)f(y)dy\Big|.
\end{eqnarray*}
\end{lemma}

\begin{proof} The estimate (\ref{eqn2.1}) was proved in \cite[p. 396]{wat}. To prove the estimate (\ref{eqn2.2}), we  will employ the
ideas used in \cite{drf}, with appropriate modifications. Let $S_l^jf(x)=\big(K_j-K_j*\psi_{j-l}\big)*f(x)$. Let
$\omega\in C^{\infty}_0(\mathbb{R}^d)$ such that $${\rm supp}\,\omega\subset\{x\in
\mathbb{R}^d:\, |x|\leq 1\},\,\,\omega(x)\equiv
1\,\,\hbox{if}\,\,|x|\leq 1/2.$$ For each integer $k$, let
$\Psi_k\in \mathcal{S}(\mathbb{R}^d)$ such that
$\widehat{\Psi}_k(\xi)=\omega(2^k\xi).$    For each fixed $k\in
\mathbb{Z}$, write
\begin{eqnarray*}\sum_{j=k}^{\infty}S_l^j*f(x)&=&
\Psi_k*\big(T_{\Omega}f-T_lf\big)(x)-\Psi_k*\Big(\sum_{j=-\infty}^{k-1}S_l^j*f\Big)(x)\\
&&+\sum_{j=k}^{\infty}(\delta-\Psi_k)*S_l^j*f(x)\\
&=&{\rm I}_l^kf(x)+{\rm II}_l^kf(x)+{\rm III}_l^kf(x),
\end{eqnarray*}
with $\delta$ the Dirac distribution. It is obvious that
$$\big|{\rm I}_l^kf(x)\big|\lesssim M\big(T_{\Omega}f-T_{l}f\big)(x),$$
and so by (\ref{eqn2.1}),
$$\big\|\sup_{k\in\mathbb{Z}}|{\rm
I}_l^kf|\big\|_{L^{2}(\mathbb{R}^d)}\lesssim\|T_{\Omega}f-T_{l}f\|_{L^2(\mathbb{R}^d)}\lesssim
2^{-\kappa l}\|f\|_{L^{2}(\mathbb{R}^d)}.$$ Now write
$$\sup_{k\in \mathbb{Z}}|{\rm II}_l^kf(x)|\lesssim
\Big(\sum_{u=-\infty}^{\infty}\Big|\Psi_u*\sum_{j=-\infty}^{u-1}S_l^j*f(x)\Big|^2\Big)^{1/2}.
$$Note that for any $\xi\in \mathbb{R}^d$,
\begin{eqnarray*}
\Big|\omega(2^u\xi)\sum_{j=-\infty}^{u-1}\widehat{K_j}(\xi)
\big(\widehat{\psi}(2^{j-l}\xi)-1\big)\Big|
&\lesssim&\Big|\omega(2^u\xi)\sum_{j=-\infty}^{u-1}|2^{j-l}\xi|\Big|\\
&\lesssim&2^{-l}\omega(2^u\xi)|2^u\xi|.
\end{eqnarray*}
It then follows from  Plancherel's theorem that
\begin{eqnarray*}
&&\big\|\sup_{k\in \mathbb{Z}}|{\rm
II}_l^kf|\big\|_{L^2(\mathbb{R}^d)}^2=\sum_{u=-\infty}^{\infty}\Big\|\Psi_u*\sum_{j=-\infty}^{u-1}S_l^j*f\Big\|_{L^2(\mathbb{R}^d)}^2\\
&&\qquad=\sum_{u=-\infty}^{\infty}\int_{\mathbb{R}^d}\Big|\sum_{j=-\infty}^{u-1}
\widehat{K_{j}}(\xi)
\big(\widehat{\psi}(2^{j-l}\xi)-1\big)\Big|^2|\omega(2^u\xi)\widehat{f}(\xi)|^2{\rm d}\xi\\
&&\qquad\lesssim
2^{-2l}\int_{\mathbb{R}^d}\sum_{u=-\infty}^{\infty}|\omega(2^u\xi)|^2|2^u\xi|^2\big|\widehat{f}(\xi)\big|^2\,{\rm
d}\xi.
\end{eqnarray*}
Recalling that $|\omega(x)|\lesssim \min\{1,\,|x|^{-2}\}$, we deduce that
\begin{eqnarray*}
\sum_{u=-\infty}^{\infty}|\omega(2^u\xi)|^2|2^u\xi|^2\lesssim\sum_{u:\, |2^u\xi|\leq 1}|2^u\xi|^2+\sum_{u:\,|2^u\xi|>1}|2^u\xi|^{-2}\lesssim 1.
\end{eqnarray*}
This, in turn, implies that
\begin{eqnarray*}\big\|\sup_{k\in \mathbb{Z}}|{\rm
II}_l^kf|\big\|_{L^2(\mathbb{R}^d)}\lesssim
2^{-l}\|f\|_{L^2(\mathbb{R}^d)}.\end{eqnarray*}

It remains to estimate term
$\sup_{k\in \mathbb{Z}}|{\rm III}_l^kf|$. write
\begin{eqnarray*}\sup_{k\in \mathbb{Z}}|{\rm III}_l^kf(x)|&\leq &
\sum_{j=0}^{\infty}\sup_{k\in\mathbb{Z}}\big|(\delta-\Psi_k)*S_{l}^{j+k}*f(x)\big|\\
&\lesssim&
\sum_{j=0}^{\infty}\Big(\sum_{u=-\infty}^{\infty}\Big|(\delta-\Psi_{u-j})*S_{l}^u*f(x)\Big|^2\Big)^{1/2}.
\end{eqnarray*}
As it is well known (see \cite{drf}), there exists a constant $\kappa_1\in (0,\,1)$ such that
$$
|\widehat{K_j}(\xi)|\lesssim \|\Omega\|_{L^{\infty}(S^{d-1})} |2^j\xi|^{-\kappa_1},
$$Therefore, \begin{eqnarray*}
&&\Big\|\Big(\sum_{u=-\infty}^{\infty}\Big|(\delta-\Psi_{u-j})*S_{l}^u*f(x)\Big|^2\Big)^{1/2}\Big\|_{L^2(\mathbb{R}^d)}^2\\
&&\quad=\sum_{u=-\infty}^{\infty}\int_{\mathbb{R}^d}\big|1-\omega(2^{u-l}\xi)\big|^2\Big|\widehat{K}_{u}(\xi)
\big(\widehat{\psi}(2^{u-l}\xi)-1\big)\Big|^2|\widehat{f}(\xi)|^2{\rm d}\xi\\
&&\quad\lesssim
\int_{\mathbb{R}^d}\sum_{u=-\infty}^{\infty}|1-\omega(2^{u-l}\xi)|^2|2^{u}\xi|^{-2\kappa_1}\big)\big|\widehat{f}(\xi)\big|^2\,{\rm d}\xi\\
&&\quad\lesssim 2^{-2l\kappa_1}\|f\|_{L^2(\mathbb{R}^d)}^2,
\end{eqnarray*}
since for each $\xi\in \mathbb{R}^d\backslash\{0\}$,
\begin{eqnarray*}\sum_{u=-\infty}^{\infty}|1-\psi(2^{u-l}\xi)|^2|2^{u}\xi|^{-2\kappa_1}&\lesssim&
\sum_{u:\,|2^u\xi|\geq
2^{l-1}}|2^{u}\xi|^{-2\kappa_1}\lesssim2^{-2l\kappa_1}.\end{eqnarray*} Thus,
$$\big\|\sup_{k\in \mathbb{Z}}|{\rm
III}_l^kf|\big\|_{L^2(\mathbb{R}^d)}\lesssim
2^{-\kappa_1 l}\|f\|_{L^2(\mathbb{R}^d)}.$$
Combining the estimates
for $\sup_{k\in \mathbb{Z}}|{\rm I}_l^kf|$, $\sup_{k\in \mathbb{Z}}|{\rm
II}_l^kf|$ and $\sup_{k\in \mathbb{Z}}|{\rm III}_l^kf|$ leads to our desired conclusion and then completes the proof of Lemma \ref{lem2.4}.\end{proof}

For $m\in\mathbb{N}$, let $H_m^{**}$ be the operator defined by
$$
H^{**}_{m}f(x)=\sup_{k\in\mathbb{Z}}\Big|\sum_{j\geq k}\int_{\mathbb{R}^d}\big(K_j*\psi_{j-2^m}(x-y)-K_j*\psi_{j-2^{m-1}}(x-y)\big)f(y)dy\Big|.
$$
Lemma \ref{lem2.4} tells us that for $m\in\mathbb{N}$,
\begin{eqnarray}\label{eqn2.approx}
\|H^{**}_mf\|_{L^{2}(\mathbb{R}^d)}\lesssim  2^{-\kappa 2^m}\|f\|_{L^2(\mathbb{R}^d)},
\end{eqnarray}
see  Lemma 8 in \cite{hon} for another proof of (\ref{eqn2.approx}).

\begin{lemma}\label{lem2.5} Let $l\in\mathbb{N}$, $\Omega$ be homogeneous of degree zero and have mean value zero.
Suppose that $\Omega\in L^{\infty}(S^{d-1})$. Then the operators $T_l$ and $T_l^{**}$ are all bounded from $L^{1}(\mathbb{R}^d)$ to $L^{1,\,\infty}(\mathbb{R}^d)$ with bound no more than $Cl\|\Omega\|_{L^{\infty}(S^{d-1})}$.
\end{lemma}

\begin{proof} The proof of Lemma \ref{lem2.5} was essentially given in \cite[Section 2]{ler5}, with a different definition of $T^l$. We present the proof here for the sake of self-contained. Without loss of generality, we may assume that $\|\Omega\|_{L^{\infty}(S^{d-1})}=1$. Obviously, $K^l$ satisfies the size condition that
\begin{eqnarray}\label{eq2.lem2.21}|K^l(x)|&\leq& \sum_{j\in\mathbb{Z}}\int_{\mathbb{R}^d}|K_j(x-y)||\psi_{j-l}(y)|dy\\
&\lesssim&\sum_{j\in\mathbb{Z}}\frac{1}{|x|^d}\chi_{\{2^{j-2}\leq |x|\leq 2^{j+2}\}}(x)\int_{\mathbb{R}^d}|\psi_{j-l}(y)|dy\lesssim
|x|^{-d}.\nonumber
\end{eqnarray}
On the other hand, a trivial computation leads to that
\begin{eqnarray*}
|\nabla K^l(x)|&\leq& \sum_{j\in\mathbb{Z}}\int_{\mathbb{R}^d}|K_j(y)||\nabla \psi_{j-l}(x-y)|dy\\
&\lesssim&\sum_{j\in\mathbb{Z}}2^{l-j}\frac{1}{|x|^d}\chi_{\{2^{j-2}\leq |x|\leq 2^{j+2}\}}(|x|)\int_{\mathbb{R}^d}|K_j(y)|dy\lesssim 2^l
|x|^{-d-1}.
\end{eqnarray*}
Therefore, for $x,\,y\in\mathbb{R}^d$ with $|y|\leq |x|/4$, we have that
\begin{eqnarray}\label{eq2.lem2.22}|K^l(x-y)-K^l(x)|\lesssim \frac{1}{|x|^d}\omega\big(\frac{|y|}{x}\big),
\end{eqnarray}
with
$$\omega(t)=\min\{1,\,2^lt\}.$$
Obviously,
$$\int_0^1\omega(t)\frac{dt}{t}\lesssim l.
$$
Lemma \ref{lem2.4} tells us that
\begin{eqnarray}\label{eq2.lem2.23}\|T_lf\|_{L^2(\mathbb{R}^d)}\lesssim \|f\|_{L^2(\mathbb{R}^d)}.
\end{eqnarray}
Inequalities (\ref{eq2.lem2.21})-(\ref{eq2.lem2.23}), via argument in \cite[Chap. 4.3]{gra}, implies that
$$\|T_l^{**}f\|_{ L^{1,\,\infty}(\mathbb{R}^d)}\lesssim l\|f\|_{L^1(\mathbb{R}^d)}.
$$
This completes the proof of Lemma \ref{lem2.5}.
\end{proof}

\begin{lemma}\label{lem2.6} Let $\Omega$ be homogeneous of degree zero and have mean value zero.
Suppose that $\Omega\in L^{\infty}(S^{d-1})$ and $l\in\mathbb{Z}$. Then $\mathcal{M}_{T_l^{**}}$ is bounded from $L^{1}(\mathbb{R}^d)$ to $L^{1,\,\infty}(\mathbb{R}^d)$ with bound $Cl\|\Omega\|_{L^{\infty}(S^{d-1})}$.
\end{lemma}

\begin{proof} The arguement is similar to the proof  of Lemma 2.3 in \cite{ler3}. Again we assume that $\|\Omega\|_{L^{\infty}(S^{d-1})}=1$.
Let $Q\subset \mathbb{R}^d$ be a cube, $x, \xi\in Q$, and $B_x$ be the closed ball centered at $x$ and having radius $r_B=2d\ell(Q)$. Then $3Q\subset B_x$. For $j\in \mathbb{Z}$ and $l\in \mathbb{N}$, let $T_{l,j}$ be the operator defined by  \begin{eqnarray}\label{eq2.tlj}
T_{l,j}f(x)=K_j*\psi_{j-l}*f(x).\end{eqnarray}
Write
\begin{eqnarray*}
|T_l^{**}(f\chi_{\mathbb{R}^d\backslash 3Q})(\xi)| &\leq&  |T_l^{**}(f\chi_{B_x\backslash 3Q})(\xi)| + |T_l^{**}(f\chi_{\mathbb{R}^d\backslash B_x})(x)|\\
&& + \sup_{k\in\mathbb{Z}}\Big|\sum_{j\geq k}\Big(T_{l,j}(f\chi_{\mathbb{R}^d\backslash B_x})(\xi)-T_{l,j}(f\chi_{\mathbb{R}^d\backslash B_x})(x)\Big)\Big|.
\end{eqnarray*}
Note that $$|K_j*\psi_{j-l}(y)|\lesssim 2^{-jd}\|\Omega\|_{L^{\infty}(S^{d-1})}\chi_{\{y: 2^{j-2}<|y|\leq 2^{j+2}\}},$$ we have
\begin{eqnarray}\label{eqn2.5}
&&|T_l^{**}(f\chi_{B_x\backslash 3Q})(\xi)| \leq\sum_{j\in \mathbb{Z}}\int_{B_x\backslash 3Q} |K_j*\psi_{j-l}(\xi-y)| |f(y)|dy\\
&&\quad\quad \lesssim \sum_{j: 2^j\approx \ell(Q)}2^{-jd}\int_{B_x}  |f(y)|dy \lesssim Mf(x).\nonumber
\end{eqnarray}
Let $k_0\in\mathbb{Z}$ such that $2^{k_0-1}\leq r_B< 2^{k_0}$. Since
\begin{eqnarray*}
&&\sum_{j\in \mathbb{Z}}\int_{|y-x|>r_B} K_j*\psi_{j-l}(x-y) f(y)dy\\
&&\quad\quad = \Big(\sum_{j\geq k_0+2}\int_{\mathbb{R}^d}  + \sum_{j= k_0-2}^{k_0+1}\int_{r_B<|y-x|\leq 2^{j+2}}\Big) K_j*\psi_{j-l}(x-y) f(y)dy,
\end{eqnarray*}
we have
\begin{eqnarray}\label{eqn2.6}
|T_l^{**}(f\chi_{\mathbb{R}^d\backslash B_x})(x)| \leq  |T_l^{**}(f)(x)| + Mf(x).
\end{eqnarray}
Also note that for any $k\in\mathbb{Z}$,
\begin{eqnarray*}
&&\Big|\sum_{j\geq k}\int_{|y-x|>r_B} \Big(K_j*\psi_{j-l}(\xi-y)- K_j*\psi_{j-l}(x-y)\Big)f(y)dy\Big|\\
&&\quad \lesssim\sum_{i=1}^\infty \sum_{j: 2^j\approx 2^ir_B}\int_{2^{i-1} r_B<|y-x|\leq 2^{i}r_B} 2^{-jd}\|\Omega\|_{L^{\infty}(S^{d-1})} \min\{1, \frac{r_B}{2^{j-l}}\} |f(y)|dy\\
\\
&&\quad \lesssim\sum_{i=1}^\infty \min\{1, 2^{l-i}\}  Mf(x),
\end{eqnarray*}
which yields
\begin{eqnarray*}
\sup_{k\in\mathbb{Z}}\Big|\sum_{j\geq k}\Big(T_{l,j}(f\chi_{\mathbb{R}^d\backslash B_x})(\xi)-T_{l,j}(f\chi_{\mathbb{R}^d\backslash B_x})(x)\Big)\Big| \lesssim  l Mf(x).
\end{eqnarray*}
This, together with (\ref{eqn2.5}) and (\ref{eqn2.6}), gives that
\begin{eqnarray}\label{eqn2.7}
\mathcal{M}_{T_l^{**}}f(x) \lesssim  |T_l^{**}(f)(x)|  + l Mf(x).
\end{eqnarray}
From this and Lemma \ref{lem2.5}, we get the lemma.
\end{proof}

\subsection{Proof of Theorem \ref{dingli2.1}}
We employ the ideas in the proof of Lemma 2.3 in \cite{ler5}, together with the ideas used in \cite{bhmo} and some more refined decomposition. Let $\mathscr{D}$ and $\mathscr{D}'$ be two dyadic lattices  and $\mathcal{F}$ be a finite family of cubes $Q$ from $\mathscr{D}$ such that $3Q\in \mathscr{D}'$. Set
$$M_{\lambda,\,T_{\Omega}^{**}}^{\mathcal{F}}f(x)=\Bigg\{ \begin{array}{ll}
\max_{Q\ni x,\,Q\in \mathcal{F}}\big(T_{\Omega}^{**}(f\chi_{\mathbb{R}^d\backslash 3Q})\chi_{Q}\big)^*(\lambda |Q|),\,&x\in \cup_{Q\in\mathcal{F}}Q\\
0,\,&\hbox{otherwise}.\end{array}$$
As it was pointed out in \cite[p. 1590]{ler5}, it suffices to prove (\ref{eqn2.8}) with $M_{\lambda,T_{\Omega}^{**}}$ replaced by $M^{\mathcal{F}}_{\lambda,T_{\Omega}^{**}}$.

By  homogeneity, it suffices to prove (\ref{eqn2.8}) for the case of $\alpha=1$. Let $M^{\mathscr{D'}}$ be the maximal operator defined by
$$
M^{\mathscr{D'}}f(x)=\sup_{Q\ni x,\,Q\in\mathscr{D'}}\langle |f|\rangle_Q.
$$
Now let $f\in L^1(\mathbb{R}^d)$, $0<\lambda<1$. Decompose $\{x\in\mathbb{R}^d:\,M^{\mathscr{D}'}f(x)>\lambda^{-1}\}$ as $\{x\in\mathbb{R}^d:\,M^{\mathscr{D}'}f(x)>\lambda^{-1}\}=\cup_{P\in\mathcal{P}} P$, with $P$ the maximal cubes in $\mathscr{D}'$ such that
$$
\lambda^{-1} < \frac{1}{|P|}\int_P|f(x)|dx .
$$
For each $P\in\mathcal{P}$, let
$$
b_P(x)=f(x)\chi_{P(|f|> 2^{c_1}\lambda^{-1})}(x),
$$
and $0<c_1<\frac 14$ be a constant which will be chosen later.
We decompose $f$ as $f=g+b= g+ \sum_{P\in\mathcal{P}}b_P$, where
$$
g(x)=f\chi_{\mathbb{R}^d\backslash \cup_{P\in\mathcal{P}}P}(x)+\sum_{P\in\mathcal{P}}f(x)\chi_{P(|f|\leq 2^{c_1}\lambda^{-1})}(x).
$$
We further decompose $b_P$ into $b_P=\sum_{l\in\mathbb{N}} b^l_P$, where
$$
b^l_P(x)=f(x)\chi_{P(2^{c_12^{l-1}}\lambda^{-1}<|f|\leq 2^{c_12^l}\lambda^{-1})}(x).
$$
For $l\in\mathbb{N}$, let
$$b_{P,\,1}^l(x)=\frac{1}{|P|}\int_{P(2^{c_12^{l-1}}\lambda^{-1}<|f|\leq 2^{c_12^l}\lambda^{-1})}f(y)dy\chi_{P}(x),$$
and
$$b_{P,2}^l(x)=b^l_P(x)-b_{P,1}^l(x).$$
Then
$$b(x)= \sum_{P\in\mathcal{P}}\sum_{l=1}^{\infty}(b_{P,1}^l(x)+b_{P,2}^l(x)).$$
For each fixed $l\in\mathbb{N}$, set
$$
G_1^l(x)=\sum_{P\in\mathcal{P}}b_{P,1}^l(x),\,\,G_2^l(x)=\sum_{P\in\mathcal{P}}b_{P,2}^l(x),\,G^l(x)=G_1^l(x)
+G_2^l(x).
$$
Obviously, $b(x)=\sum_{l\in\mathbb{N}}G^l(x)$.
For fixed $j\in\mathbb{Z}$, set $\mathcal{P}_j=\{P:\, P\in\mathcal{P},\,\ell(P)=2^j\}$. Let $B_{j}(x)=\sum_{P\in \mathcal{P}_j}b_P(x)$ and $B_{j,\,1}^l(x)=\sum_{P\in \mathcal{P}_j}b_{P,1}^l(x)$,
$B_{j,2}^l(x)=\sum_{P\in\mathcal{P}_j}b_{P,\,2}^l(x)$ for $l\in\mathbb{N}$, and $B_{j}^l(x)=B_{j,1}^l(x)+B_{j,2}^l(x).$
Note that $G^l(x)=\sum_{j}B_j^l(x)$, and we can see that
\begin{itemize}
\item [\rm (i)] for all cube $P\in\mathcal{P}$ and $l\in\mathbb{N}$, $\int b_{P,2}^l(x)dx=0,$ and
\begin{eqnarray}\label{eq2.b2l}
\|b_{P,2}^l\|_{L^1(\mathbb{R}^d)}\lesssim \int_{P(2^{c_12^{l-1}}\lambda^{-1}<|f(x)\leq 2^{c_12^l}\lambda^{-1})}|f(x)|dx;
\end{eqnarray}
\item[\rm (ii)] for each fixed $l$, $\|G^l\|_{L^2(\mathbb{R}^d)}^2\lesssim 2^{c_12^l}\lambda^{-1} \|f\|_{L^1(\mathbb{R}^d)};$
\item [\rm (iii)] $\big\|\sum_{l=1}^{\infty}\sum_{P\in\mathcal{P}}|b_{P,1}^l|\big\|_{L^{\infty}(\mathbb{R}^d)}\lesssim \lambda^{-1}$, and  \begin{eqnarray}\label{eq:2.1.first}&&\Big\|\sum_{l=1}^{\infty}|G_1^l|\Big\|_{L^{1}(\mathbb{R}^d)}\lesssim \|f\|_{L^{1}(\mathbb{R}^d)},\,\,\,\Big\|\sum_{l=1}^{\infty}|G_{1}^l|\Big\|_{L^{2}(\mathbb{R}^d)}^2\lesssim \lambda^{-1}\|f\|_{L^{1}(\mathbb{R}^d)}.\end{eqnarray}
\item [\rm (iv)] $\|g\|_{L^{\infty}(\mathbb{R}^d)}\lesssim \lambda^{-1}$.
\end{itemize}

Now let $l_0\in\mathbb{N}$ such that $\log(\frac{1}{\lambda})+1<\kappa l_0\leq\log(\frac{1}{\lambda})+2$, with $\kappa$ the positive constant in Lemma \ref{lem2.4}.
A straightforward computation involving Lemma \ref{lem2.2} and Lemma \ref{lem2.4} leads to that
\begin{eqnarray*}
|\{x\in\mathbb{R}^d:\,M_{\lambda/4, T_{\Omega}^{**}-T_{l_0}^{**}}^{\mathcal{F}}g(x)>1/4\}|
\lesssim\lambda^{-1}2^{-2\kappa l_0}\|g\|_{L^2(\mathbb{R}^d)}^2 \lesssim\|f\|_{L^1(\mathbb{R}^d)}.\nonumber
\end{eqnarray*}
By Lemma \ref{lem2.6} and the observation  that
$$
M_{\lambda/4,T_{l_0}^{**}}^{\mathcal{F}}g(x)\leq \mathcal{M}_{T_{l_0}^{**}}g(x),
$$
we then have
$$
|\{x\in\mathbb{R}^d:\, M_{\lambda/4,T_{l_0}^{**}}^{\mathcal{F}}g(x)> 1/4\}|\lesssim l_0\|g\|_{L^1(\mathbb{R}^d)}\lesssim
\big(1+\log \frac{1}{\lambda}\big)\|f\|_{L^1(\mathbb{R}^d)}.
$$
Therefore,
\begin{eqnarray}\label{equation2.gx} &&
|\{x\in\mathbb{R}^d:\, M_{\lambda/2,T_{\Omega}^{**}}^{\mathcal{F}}g(x)> 1/2\}|\lesssim l_0\|g\|_{L^1(\mathbb{R}^d)}\lesssim
\big(1+\log \frac{1}{\lambda}\big)\|f\|_{L^1(\mathbb{R}^d)}.
\end{eqnarray}
We now estimate $ M_{\lambda/2, T^{**}_{\Omega}}^{\mathcal{F}}b$.
Set $E=\cup_{P\in\mathcal{P}}36dP$ and $E^*=\{x\in\mathbb{R}^d:\,M^{\mathscr{D}}\chi_{E}(x)>\frac{\lambda}{32}\}$. We then have that
$$
|E^*|\lesssim \lambda^{-1}|E|\lesssim\|f\|_{L^1(\mathbb{R}^d)}.
$$
If we can prove that
\begin{eqnarray}\label{eqn2.11}
&&|\{x\in\mathbb{R}^d\backslash E^*:\, M_{\lambda/2, T^{**}_{\Omega}}^{\mathcal{F}}b(x)> 1/2\}|\\
&&\quad\lesssim
 (1+\log(\frac 1\lambda))\int_{ \mathbb{R}^d }|f(x)|dx+\int_{ \mathbb{R}^d }\Phi(|f(x)|)dx,\nonumber
\end{eqnarray}
our desired estimate (\ref{eqn2.8})  then follows from   (\ref{equation2.gx})--(\ref{eqn2.11}) directly.

We now prove (\ref{eqn2.11}).
For each $l\in\mathbb{N}$ and $j\in\mathbb{Z}$, let $T_{l,j}$ be defined as (\ref{eq2.tlj}).
Recall $T_{\Omega,j}(f)(x)=K_j*f(x)$. Now let $Q\in\mathcal{F}$, and
$${\rm D}_{21,Q} b(x)=\sum_{l=1}^{\infty}\sup_{k}\Big|\sum_{j>k}(T_{\Omega,j}-T_{2^{l+N_0},j})(G^l\chi_{\mathbb{R}^d\backslash 3Q})(x)
\Big|,$$
$${\rm D}_{22,Q}b(x)=\sup_{k}\Big|\sum_{j>k}\sum_{l=1}^{\infty}\sum_{m=1}^{l}(T_{2^{m+N_0},j}-T_{2^{m+N_0-1},j})(\chi_{\mathbb{R}^d\backslash 3Q}G_1^l)(x)\Big|,$$
$${\rm D}_{23,Q}b(x)=\sup_{k}\Big|\sum_{j>k}T_{2^{N_0},j}\big(\chi_{\mathbb{R}^d\backslash 3Q}b\big)(x)\Big|,$$
$${\rm D}_{24,Q}b(x)=\sup_{k}\Big|\sum_{j>k}\sum_{l=1}^{\infty}(T_{2^{l+N_0},j}-T_{2^{N_0},j})
(G_2^l\chi_{\mathbb{R}^d\backslash 3Q})(x)\Big|,$$
where $N_0\in\mathbb{N}$ is a constant which depends on $\lambda$ and will be chosen later. Then
\begin{eqnarray*}\Big(T^{**}_{\Omega}(b\chi_{\mathbb{R}^d\backslash 3Q})\chi_Q\Big)^*(\frac{\lambda|Q|}{2})&\leq&\big(({\rm D}_{21,Q}b)\chi_{Q}\big)^*(\frac{\lambda|Q|}{8})+
\big(({\rm D}_{22,Q}b)\chi_{Q}\big)^*\big(\frac{\lambda|Q|}{8}\big)\\
&&+\big(({\rm D}_{23,Q}b)\chi_{Q}\big)^*\big(\frac{\lambda|Q|}{8}\big)+\big(({\rm D}_{24,Q}b)\chi_{Q}\big)^*\big(\frac{\lambda|Q|}{8}\big),
\end{eqnarray*}
and
\begin{eqnarray*}M_{\lambda/2, T_{\Omega}^{**}}^{\mathcal{F}}b(x)\leq \sum_{i=1}^4E_i(x),
\end{eqnarray*}
where$$E_i(x)=\Bigg\{ \begin{array}{ll}
\max_{Q\ni x,\,Q\in \mathcal{F}}\big((D_{2i,Q}b)\chi_{Q}\big)^*\big(\frac{\lambda |Q|}{8}\big),\,&x\in \cup_{Q\in\mathcal{F}}Q\\
0,\,&\hbox{otherwise}.\end{array}$$
Now set
$$F_1(x)=M_2
\Big(\sum_{l=1}^{\infty}\sup_{k}\big|\sum_{j>k}(T_{2^{l+N_0},j}-T_{\Omega,j})G^l\big|\Big)(x),$$
$$F_2(x)=\sup_{Q\ni x}\Big(\frac{1}{|Q|}\int_Q\Big(\sum_{l=1}^{\infty}\sup_{k}\big|\sum_{j>k}(T_{2^{l+N_0},j}-T_{\Omega,j})(
\chi_{Q}G^l)(y)\big|\Big)^2dy\Big)^{1/2},$$
where $M_2f(x)=\big(M(|f|^2)(x)\big)^{1/2}$. Recall that $M_2$ is bounded from $L^2(\mathbb{R}^d)$ to $L^{2,\,\infty}(\mathbb{R}^d)$, and the fact that $2^{l+N_0}\geq 2^l+2^{N_0}$, it follows from Lemma \ref{lem2.4} that
\begin{eqnarray*}\|F_1\|_{L^{2,\,\infty}(\mathbb{R}^d)}&\lesssim & \sum_{l=1}^{\infty}\big\|\sup_{k}\big|\sum_{j>k}(T_{2^{l+N_0},j}-T_{\Omega,j})G^l\big|\big \|_{L^2(\mathbb{R}^d)}\\
&\lesssim&\sum_{l=1}^{\infty}2^{-\kappa 2^{l+N_0}}\|G^l\|_{L^2(\mathbb{R}^d)}\\
&\lesssim&\sum_{l=1}^{\infty}2^{-\kappa 2^{l+N_0}}\left(2^{c_12^l}\lambda^{-1}\|f\|_{L^1(\mathbb{R}^d)}\right)^{1/2}\\
&\lesssim&2^{-\kappa 2^{N_0}}\left(\lambda^{-1}\|f\|_{L^1(\mathbb{R}^d)}\right)^{1/2},
\end{eqnarray*}
if we choose $0<c_1<\kappa$. Another application of Lemma \ref{lem2.4} leads to that
\begin{eqnarray*}
F_2(x)&\leq &\sum_{l=1}^{\infty}\sup_{Q\ni x}\Big(\frac{1}{|Q|}\int_Q\Big(\sup_{k}\big|\sum_{j>k}(T_{2^{l+N_0},j}-T_{\Omega,j})(
\chi_{Q}G^l)(y)\big|\Big)^2dy\Big)^{1/2}\\
&\lesssim &\sum_{l=1}^{\infty}2^{-\kappa 2^{l+N_0}}M_2G^l(x).
\end{eqnarray*}
Observing that for a constant $\sigma>0$,
$$\Big\|\sum_{k=1}^{\infty}f_k\Big\|_{L^{2,\,\infty}(\mathbb{R}^d)}^2\lesssim \sum_{k=1}^{\infty}2^{\sigma k}\|f_k\|_{L^{2,\,\infty}(\mathbb{R}^d)}^2,$$
and $2^{-2\kappa 2^{l+N_0}} \le 2^{-2\kappa2^l}2^{-2\kappa2^{N_0}}$, we then get  that
\begin{eqnarray*}\|F_2\|_{L^{2,\,\infty}(\mathbb{R}^d)}^2&\lesssim & \sum_{l=1}^{\infty}2^{-2\kappa2^{l+N_0}}
2^{c_1l}\|M_2G^l\|^2_{L^{2,\,\infty}(\mathbb{R}^d)}\\
&\lesssim&\sum_{l=1}^{\infty}2^{-2\kappa 2^{l+N_0}} 2^{2c_12^l}\lambda^{-1}\|f\|_{L^1(\mathbb{R}^d)} \\
&\lesssim&2^{-2\kappa 2^{N_0}}\left(\lambda^{-1}\|f\|_{L^1(\mathbb{R}^d)}\right).
\end{eqnarray*}
On the other hand, it follows from  Chebyshev's inequality that for each cube $Q$ containing $x$,
\begin{eqnarray*}\big(({\rm D}_{21,Q}b)\chi_{Q}\big)^*\big(\frac{\lambda|Q|}{8}\big)&\lesssim& \Big(\sum_{l=1}^{\infty}\sup_{k}\Big|\sum_{j>k}(T_{\Omega,j}-T_{2^{l+N_0},j})G^l\Big|\chi_{Q}\Big)^*(\frac{\lambda|Q|}{16})\\
&& +\big(\sum_{l=1}^{\infty}\sup_{k}\Big|\sum_{j>k}(T_{\Omega,j}-T_{2^{l+N_0},j})(G^l\chi_{3Q})\Big|\chi_{Q}\Big)^*(\frac{\lambda|Q|}{16})\\
&\lesssim& \lambda^{-\frac{1}{2}}\big(F_1(x)+F_2(x)\big),
\end{eqnarray*}
and so
\begin{eqnarray}\label{eqn2.12}
|\{x\in\mathbb{R}^d:\, E_1(x)>1/8\}|&\lesssim&\lambda^{-1}\big(\|F_1\|_{L^{2,\,\infty}(\mathbb{R}^d)}^2+
\|F_2\|_{L^{2,\,\infty}(\mathbb{R}^d)}^2\big)\\
&\lesssim&\|f\|_{L^1(\mathbb{R}^d)},\nonumber
\end{eqnarray}
if we choose $N_0=\lfloor \log \big(\frac{1}{\kappa}\big(1+\log (\frac{1}{\lambda})\big)\rfloor+1$.

The estimate for term $E_2$ is similar to the estimate for $E_1$. In fact, let
$$F_3(x)=M_2
\Big(\sup_{k}\Big|\sum_{j>k}\sum_{l=1}^{\infty}\sum_{m=1}^l(T_{2^{m+N_0},j}-T_{2^{m+N_0-1},j})G^l_1\Big|\Big)(x),$$
and
$$F_4(x)=\sup_{Q\ni x}\Big(\frac{1}{|Q|}\int_Q\Big(\sup_{k}\Big|\sum_{j>k}\sum_{l=1}^{\infty}\sum_{m=1}^l(T_{2^{m+N_0},j}-
T_{2^{m+N_0-1},j})
(
\chi_{Q}G^l_1)(y)\Big|\Big)^2dy\Big)^{1/2},$$
Then $$E_2(x)\leq F_3(x)+F_4(x).$$
Observe that
$$F_3(x)\leq M_2\Big(\sum_{m=1}^{\infty}\sup_{k}
\Big|\sum_{j>k}(T_{2^{m+N_0},j}-T_{2^{m+N_0-1},j})\big(\sum_{l=m}^{\infty}G^l_1\big)\Big|\Big)(x),$$
it follows from (\ref{eqn2.approx}) and (\ref{eq:2.1.first}) that
\begin{eqnarray*}\|F_3\|_{L^{2,\,\infty}(\mathbb{R}^d)}
&\lesssim&\sum_{m=1}^{\infty}\big\|H_{m+N_0}^{**}\big(\sum_{l=m}^{\infty}G^l_1\big)\big\|_{L^2(\mathbb{R}^d)}\\
&\lesssim&\sum_{m=1}^{\infty}2^{-\kappa 2^{m+N_0-1}}\big(\lambda^{-1}\|f\|_{L^1(\mathbb{R}^d)}\big)^{1/2}\\
&\lesssim&2^{-\kappa 2^{N_0}}\left(\lambda^{-1}\|f\|_{L^1(\mathbb{R}^d)}\right)^{1/2}.
\end{eqnarray*}
Since
\begin{eqnarray*}F_4(x)&\leq &\sum_{m=1}^{\infty}\sup_{Q\ni x}\Big\{\frac{1}{|Q|}\int_Q\Big(H_{m+N_0}^{**}\big(\sum_{l=m}^{\infty}G^l_1
\chi_{Q}\big)(y)\Big)^2dy\Big\}^{\frac{1}{2}}\\
&\lesssim &\sum_{m=1}^{\infty}2^{-\kappa 2^{m+N_0-1}}M_2
\big(\sum_{l=m}^{\infty}G^l_1\big)(x),
\end{eqnarray*}
we then obtain from (\ref{eq:2.1.first}) that
\begin{eqnarray*}\|F_4\|_{L^{2,\,\infty}(\mathbb{R}^d)}&\lesssim & \sum_{m=1}^{\infty}2^{-\kappa2^{m+N_0-1}}
2^{\kappa m}\Big\|M_2\big(\sum_{l=m}^{\infty}G^l_1\big)\Big\|_{L^{2,\,\infty}(\mathbb{R}^d)}\\
&\lesssim&2^{-\kappa 2^{N_0}}\left(\lambda^{-1}\|f\|_{L^1(\mathbb{R}^d)}\right)^{1/2}.
\end{eqnarray*}
Therefore,
\begin{eqnarray*}
|\{x\in\mathbb{R}^d:\, E_2(x)>1/8\}|
&\lesssim& \lambda^{-2}2^{-2\kappa2^{N_0}}\|f\|_{L^{1}(\mathbb{R}^d)}\lesssim\|f\|_{L^1(\mathbb{R}^d)}.
\end{eqnarray*}

To estimate $E_3$, we apply Lemma \ref{lem2.6} to obtain that
\begin{eqnarray*}
|\{x\in\mathbb{R}^d:\, E_3(x)>1/8\}|&\lesssim& |\{x\in\mathbb{R}^d:\,\mathcal{M}_{T_{2^{N_0}}^{**}}b(x)> 1/8\}|\\
&\lesssim & 2^{N_0}\|b\|_{L^{1}(\mathbb{R}^d)}\\
&\lesssim& (1+\log(\frac 1\lambda))\|f\|_{L^1(\mathbb{R}^d)},
\end{eqnarray*}
since $2^{N_0}\lesssim 1+\log(\frac 1\lambda)$.

With the estimates for $E_1$, $E_2$ and $E_3$ above, the proof of (\ref{eqn2.11}) is now reduced to prove that
\begin{eqnarray}\label{eqn2.final}
&&|\{x\in\mathbb{R}^d\backslash E^*:\, E_4(x)>1/8\}|
\lesssim (1+\log(\frac 1\lambda))\int_{\mathbb{R}^d}|f(x)|dx+\int_{\mathbb{R}^d}\Phi(|f(x)|)dx.
\end{eqnarray}
This will be given in the next subsection.
\qed
\subsection{Proof of (\ref{eqn2.final})}
To begin with, we give two useful lemmas.
\begin{lemma}\label{lemm2.8}
Let $\{H_j\}$ be a sequence of kernels supported in $\{x:\, 2^{j-2}\leq |x|\leq 2^{j+2}\}$ and for each $n\in\mathbb{N}$,
\begin{eqnarray}\label{eqn2.15}
\sup_{0\leq \tau\leq n}\sup_{j\in\mathbb{Z}}r^{d+\tau}\Big|(\frac{\partial}{\partial r})^\tau H^j(r\theta)\Big|\leq C_n
\end{eqnarray}
uniformly in $\theta\in {S}^{d-1}$ and $r>0$. Then for  $s>3$ and positive constants $\delta_1$ and $\delta_2$,  one can split $H_j=\Gamma^j_s+(H^j-\Gamma^j_s)$ such that the following properties hold.

(1) Let $Q\in\mathcal{Q}$, a collection of pairwise disjoint dyadic cubes, denote by $f_Q$ an integrable function supported in $Q$ with $\int |f_Q|dx\leq \alpha |Q|$. Let $F_m=\sum_{Q\in\mathcal{Q}, \ell(Q)=2^m}f_Q$. Then
$$
\|\sum_j\Gamma^j_s*F_{j-s}\|_{L^2(\mathbb{R}^d)}^2\lesssim C_0^22^{-2s\delta_1}\alpha \sum_Q\|f_Q\|_{L^1(\mathbb{R}^d)}.
$$

(2) Let $Q\in\mathcal{Q}$ and $\ell(Q)=2^{j-s}$, let $b_Q$ an integrable function supported in $Q$ with $\int b_Q dx=0$. Then
$$
\|(H^j-\Gamma^j_s)*b_{Q}\|_{L^1(\mathbb{R}^d)}^2\lesssim (C_0+C_{8d})2^{-s\delta_2}\|b_Q\|_{L^1(\mathbb{R}^d)}.
$$
\end{lemma}
Lemma \ref{lemm2.8} is a combination of Lemma 2.1 and Lemma 2.2 in \cite{se}.

Denote by $K^j_m=K_j*\psi_{j-2^m}$,  and $H^j_m=K^j_{m}-K^j_{m-1}$ for $j\in\mathbb{Z}$ and $m\in\mathbb{N}$.  It is easy to see that $H^j_m$ satisfies the assumption in Lemma \ref{lemm2.8} with the bound in (\ref{eqn2.15}), with $C_n$ a constant independent of $m$, see \cite[Section 6]{hon}.

\begin{lemma} \label{lemm2.9}
Let $\Omega$ be homogeneous of degree, and $\Omega\in L^{\infty}(S^{d-1})$. There exist  a sequence of non-negative function $\{v_i\}$ with $\sup_i \|v_i\|_{L^1}\lesssim 1$, and a constant $ \delta_3\in (0,\,1)$, such that for $m,\,s\in\mathbb{N}$,
\begin{eqnarray*}
\sup_{k\in\mathbb{Z}}\Big|\sum_{j\geq k}H^j_m*B^l_{j-s,2}\Big|&\lesssim_d&\sum_{r=0}^{s2^{l+2}-1} M\Big(\sum_{j=r\hbox{mod}(s2^{l+2})}H_m^{j} *B^l_{j-s,2}\Big)\\
&&\quad + s2^l 2^{-\delta_3 s}\sum_{j\in\mathbb{Z}}\sum_{P\in\mathcal{P}_{j-s}} v_j*|b^l_{P,2}|
\end{eqnarray*}

\end{lemma}
Lemma \ref{lemm2.9} is Lemma 2.5 in \cite{hon}.

Obviously,  ${\rm supp}\,K^j_m\subset\{2^{j-2}\leq |x|\leq 2^{j+2}\} $ for any $m\in\mathbb{N}$. Recall that $E=\cup_{P\in\mathcal{P}}36dP$, we have
\begin{eqnarray}\label{eqn2.13}
\sum_{j\geq k}\sum_{l=1}^\infty \sum_{s<3} (K^j_{l+N_0}-K^j_{N_0})* B^l_{j-s,2}(x) = 0,  \end{eqnarray}
for any $x\in \mathbb{R}^d\backslash E^*$ and $k\in\mathbb{Z}$. Moreover, for all $x\in\mathbb{R}^d\backslash E^*$, $Q\in \mathcal{F}$ and $x\in Q$, we know $|Q \cap E| \leq \frac{\lambda |Q|}{32}$. Thus, for $x\in \mathbb{R}^d\backslash E^*$, we obtain from (\ref{eqn2.13}) that
$$\Big(\sum_{j\geq k}\sum_{l=1}^\infty \sum_{s<3} (K^j_{l+N_0}-K^j_{N_0})*(B^l_{j-s,2}\chi_{\mathbb{R}^d\backslash 3Q})\Big)^*(\frac{\lambda|Q|}{16}) = 0,
$$
and
\begin{eqnarray*}
{E}_4(x)&\leq& \max_{x\in Q\atop Q\in\mathcal{F}}\Big(\sup_{k}\Big|\sum_{j\geq k}\sum_{l=1}^\infty \sum_{s=3}^\infty (K_{l+N_0}^{j}-K_{N_0}^{j})* (B^l_{j-s,2}\chi_{\mathbb{R}^d\setminus 3Q})\chi_Q\Big|\Big)^*(\frac{\lambda |Q|}{16})\\
&\leq& \sum_{j\in\mathbb{Z}}\sum_{l=1}^\infty \sum_{s=3}^{c_2l+N_{\lambda}} \mathcal{M}_{(T_{2^{l+N_0},j}-T_{2^{N_0},j}) }(B^l_{j-s,2})(x)\\
& + &\max_{x\in Q\atop Q\in\mathcal{F}}\Big(\sup_{k}\Big|\sum_{j\geq k}\sum_{l=1}^\infty \sum_{s=c_2l+N_{\lambda}+1}^\infty (K_{l+N_0}^{j}-K_{N_0}^{j})* (B^l_{j-s,2}\chi_{\mathbb{R}^d\setminus 3Q})\chi_Q\Big|\Big)^*(\frac{\lambda |Q|}{32})\\
&=:&{ E}_{41}(x)+{ E}_{42}(x),
\end{eqnarray*}
where $c_2=100\lfloor\delta^{-1}\rfloor$, $N_{\lambda}=20\lfloor\delta^{-1}(1+\log (\frac{1}{\lambda})\rfloor $, with $\delta=\min\{\delta_1,\,\delta_2,\,\delta_2\}$, and $\delta_1$, $\delta_2$, $\delta_3$   the constants appeared in Lemma \ref{lemm2.8} and Lemma \ref{lemm2.9}.

We first consider ${\rm E}_{41}(x)$. Repeating the proof of Lemma 2.6 in \cite{ler5}, we can verify that for any $j\in\mathbb{Z}$ and $m\geq 0$, the operator $\mathcal{M}_{T_{m,j}}$ is bounded on $L^1(\mathbb{R}^d)$  with bound depending on $d$ and $C\|\Omega\|_{L^{\infty}({S}^{d-1})}$, but not on $j$ and $m$. A straightforward computation leads to that
\begin{eqnarray}\label{eqn2.14}
\|{ E}_{41}\|_{L^1(\mathbb{R}^d)}&\leq &\sum_{j\in\mathbb{Z}}\sum_{l=1}^\infty \sum_{s=3}^{c_2l+N_{\lambda}} \|(\mathcal{M}_{T_{2^{l+N_0},j}}+\mathcal{M}_{T_{2^{N_0},j}})(B^l_{j-s,2})\|_{L^1(\mathbb{R}^d)} \\
&\lesssim&\sum_{l=1}^\infty\sum_{s=3}^{c_2l+N_{\lambda}}\sum_{P\in\mathcal{P}} \|b^l_{P,2}\|_{L^1(\mathbb{R}^d)} \nonumber\\
&\lesssim& \big(1+\log  (\frac{1}{\lambda})\big)\int_{\mathbb{R}^d}|f(x)|dx+\int_{\mathbb{R}^d}\Phi(|f(x)|)dx,\nonumber
\end{eqnarray}
where in the lat inequality, we have invoked the inequality (\ref{eq2.b2l}).

Now we turn our attention to term ${\rm E}_{42}$. As in [13, p. 1594], we write the set $\{x:\, E_{42}(x)>1/16\}$ as   the union of maximal pairwise disjoint cubes $Q_i$, which can be selected into two disjoint families $\mathcal{A}_1$ and $\mathcal{A}_2$: for each $Q_i\in\mathcal{A}_1$,
$$
|Q_i|<\frac{64}{\lambda}\Big|\Big\{x\in Q_i: \, \sup_{k}\Big|\sum_{j\geq k}\sum_{l=1}^\infty \sum_{s=c_2l+N_{\lambda}+1}^\infty \sum_{n=1}^l H^j_{n+N_0}* B^l_{j-s,2}\Big|>\frac 1{16}\Big\}\Big|;
$$
and each $Q_i\in\mathcal{A}_2$,
$$
|Q_i|<\frac{64}{\lambda}\Big|\Big\{x\in Q_i: \, \sup_{k}\Big|\sum_{j\geq k}\sum_{l=1}^\infty \sum_{s=c_2l+N_{\lambda}+1}^\infty \sum_{n=1}^l H^j_{n+N_0}* (B^l_{j-s,2}\chi_{3Q_i})\Big|>\frac 1{16}\Big\}\Big|.
$$

By Lemma \ref{lemm2.9}, one has that
\begin{eqnarray*}
&&\sup_{k}\Big|\sum_{j\geq k}\sum_{l=1}^\infty \sum_{s=c_2l+N_{\lambda}+1}^\infty \sum_{n=1}^l H^j_{n+N_0}* B^l_{j-s,2}\Big|\\
&&\quad\quad \lesssim
\sum_{l=1}^\infty \sum_{s=c_2l+N_{\lambda}+1}^\infty \sum_{n=1}^l \sum_{r=0}^{s2^{l+2}-1} M\Big(\sum_{j=r\hbox{mod}(s2^{l+2})}H_{n+N_0}^{j} *B^l_{j-s,2}\Big)\nonumber\\
&&\quad\quad\quad + \sum_{l=1}^\infty \sum_{s=c_2l+N_{\lambda}+1}^\infty \sum_{n=1}^l s2^l 2^{-\delta s}\sum_{j\in\mathbb{Z}}\sum_{P\in\mathcal{P}_{j-s}} v_j*|b^l_{P,2}|=: {\rm I}(x)+{\rm II}(x).\nonumber
\end{eqnarray*}
By Chebyshev's inequality, we have,
\begin{eqnarray*}
\|{\rm II}\|_{L^{1, \infty}(\mathbb{R}^d)}&\lesssim&
 \sum_{l=1}^\infty \sum_{s=c_2l+N_{\lambda}+1}^\infty \sum_{n=1}^l s2^l 2^{-\delta s} \sum_{j\in\mathbb{Z}}\sum_{P\in\mathcal{P}_{j-s}} \|v_j\|_{L^{1}(\mathbb{R}^d)}\|b^l_{P,2}\|_{L^{1}(\mathbb{R}^d)}\\
&\leq& \sum_{l=1}^\infty l2^l\sum_{s=c_2l+N_{\lambda}+1}^\infty  s 2^{-\delta s} \sum_{P\in\mathcal{P}} \|b^l_{P}\|_{L^{1}(\mathbb{R}^d)}\\
&\lesssim&  2^{-N_{\lambda}\delta/4}\|f\|_{L^{1}(\mathbb{R}^d)}.
\end{eqnarray*}

To deal with term ${\rm I}$,  we use Lemma \ref{lemm2.8}  and decompose $H_{n+N_0,k}^{j}$ as
$$H_{n+N_0}^{j}=(H_{n+N_0}^{j}-\Gamma^j_{n+N_0,s})+\Gamma^j_{n+N_0, s},$$
where $\Gamma^j_{n+N_0, s}$ enjoy the properties of $\Gamma^j_s$ in Lemma \ref{lemm2.8}. Write
\begin{eqnarray*}
{\rm I}(x)&\lesssim&
\sum_{l=1}^\infty \sum_{s=c_2l+N_{\lambda}+1}^\infty \sum_{n=1}^l\sum_{r=0}^{s2^{l+2}-1} M\Big(\sum_{j=r\hbox{mod}(s2^{l+2})}(\Gamma^j_{n+N_0,s}-H_{n+N_0}^{j}) *B^l_{j-s,2}\Big)\nonumber\\
&& +\sum_{l=1}^\infty \sum_{s=c_2l+N_{\lambda}+1}^\infty \sum_{n=1}^l\sum_{r=0}^{s2^{l+2}-1} M\Big(\sum_{j=r\hbox{mod}(s2^{l+2})}\Gamma_{n+N_0,s}^{j} *B^l_{j-s,2}\Big)\\
&=:& {\rm I}_1(x) + {\rm I}_2(x).
\end{eqnarray*}
Recalling  that  $M$ is bounded from $L^1(\mathbb{R}^d)$ to $L^{1, \infty}(\mathbb{R}^d)$, we obtain from (2) of Lemma \ref{lemm2.8} that
\begin{eqnarray*}
\|{\rm I}_1\|_{L^{1, \infty}(\mathbb{R}^d)}&\lesssim&
 \sum_{l=1}^\infty l^2\sum_{n=1}^l\sum_{s=c_2l+N_{\lambda}+1}^\infty s^2\sum_{r=0}^{s2^{l+2}-1}r^2\\
 &&\times \Big\|\sum_{j=r\hbox{mod}(s2^{l+2})}(\Gamma^j_{n+N_0,s}-H_{n+N_0}^{j}) *B^l_{j-s,2}\Big\|_{L^{1}(\mathbb{R}^d)}\nonumber\\
&\lesssim& \sum_{l=1}^\infty l^3\sum_{s=c_2l+N_{\lambda}+1}^\infty (s2^{l+2})^3s^22^{-\delta s}\sum_{P\in\mathcal{P}} \|b_P^l\|_{L^{1}(\mathbb{R}^d)} \\
&\lesssim& 2^{-N_{\lambda}\delta/4}\|f\|_{L^{1}(\mathbb{R}^d)}
\end{eqnarray*}
On the other hand, it follows from Lemma \ref{lemm2.8}, the $L^2(\mathbb{R}^d)$ boundedness of $M$ and  H\"older's inequality that
\begin{eqnarray*}
\|{\rm I}_2\|_{L^{2, \infty}(\mathbb{R}^d)}&\lesssim& \sum_{l=1}^\infty l^3\sum_{n=1}^ln^2 \sum_{s=c_2l+N_{\lambda}+1}^\infty s^3\sum_{r=0}^{s2^{l+2}-1} r^3\Big\|\sum_{j=r\hbox{mod}(s2^{l+2})}\Gamma^j_{n+N_0,s} *B^l_{j-s,2}\Big\|_{L^{2}(\mathbb{R}^d)}\nonumber\\
&\lesssim& \sum_{l=1}^\infty l^6 \sum_{s=c_2l+N_{\lambda}+1}^\infty (s2^{l})^{4}2^{-\delta s/2} \Big(\sum_{P\in\mathcal{P}} \|b_P^l\|_{L^{1}(\mathbb{R}^d)}\Big)^{1/2}\\
&\lesssim&  2^{-N_{\lambda}\delta/8}\|f\|^{1/2}_{L^{1}(\mathbb{R}^d)}
\end{eqnarray*}
Combining the estimate for terms ${\rm I}$ and ${\rm II}$, we obtain that
\begin{equation}\label{eqn2.17}
\sum_{Q_i\in\mathcal{A}_1}|Q_i|\lesssim \frac 1\lambda 2^{-N_{\lambda}\delta/8}\|f\|_{L^{1}(\mathbb{R}^d)}\lesssim \|f\|_{L^1(\mathbb{R}^d)}.
\end{equation}

As it was pointed out in \cite[p. 1596]{ler5}, for each cube $Q_i\in\mathcal{A}_2$, we have that
\begin{eqnarray*}
|Q_i|&\leq&\frac{64}{\lambda}\Big|\Big\{x\in Q_i: \, \sup_{k}\big|\sum_{j\geq k}\sum_{l=1}^\infty \sum_{s=c_2l+N_{\lambda}+1}^\infty \sum_{n=1}^l H^j_{n+N_0}* (B^l_{j-s,2}\chi_{3Q_i})\big|>\frac 1{16}\Big\}\Big|\\
&\le & \frac{64}{\lambda}\Big|\Big\{x\in \mathbb{R}^d: \, \sup_{k}\big|\sum_{j\geq k}\sum_{l=1}^\infty \sum_{s=c_2l+N_{\lambda}+1}^\infty \sum_{n=1}^l H^j_{n+N_0}* B^{l,i}_{j-s,2}\big|>\frac 1{16}\Big\}\Big|
\end{eqnarray*}
where and in the following,
$$B^{l,i}_{j-s,2}(x)=\sum_{P\in \mathcal{P}_{j-s},\,P\subset 3Q_i}b_{P,2}^l(x).
$$
As in the proof of (\ref{eqn2.17}), we have that for each $Q_i\in \mathcal{A}_2$,
$$|Q_i|\lesssim \frac 1\lambda  2^{-N_{\lambda}\delta/8}\int_{3Q_i}|f(x)|dx\lesssim \int_{3Q_i}|f(x)|dx,
$$
which implies that
\begin{eqnarray}\label{eq2.21}
&&\big|\cup_{Q_i\in\mathcal{A}_2}Q_i\big|\le \big|\{x\in\mathbb{R}^d:\,Mf(x)>C_d\}\big|\lesssim  \|f\|_{L^{1}(\mathbb{R}^d)}.
\end{eqnarray}
Combining the estimates (\ref{eqn2.14})-(\ref{eq2.21}) leads to that
\begin{eqnarray*}
|\{x\in\mathbb{R}^d:\, E_4(x)>1/8\}|&\lesssim&|\{x\in\mathbb{R}^d:\, E_{41}(x)>1/16\}|\\
&&+|\{x\in\mathbb{R}^d:\, E_{42}(x)>1/16\}|\\
 &\lesssim&\big(1+\log  (\frac{1}{\lambda})\big)\int_{\mathbb{R}^d}|f(x)|dx+\int_{\mathbb{R}^d}\Phi(|f(x)|)dx.
\end{eqnarray*}
This completes the proof of (\ref{eqn2.final}).

\section{proof of Theorem \ref{dingli1.3}}
We begin with a preliminary lemma.
\begin{lemma}\label{lem3.1}
Let $0\leq\gamma_1,\,\gamma_2\leq 1$ and $T$  a sublinear operator. Suppose that for all $\lambda\in (0,\,1)$,
\begin{eqnarray*}|\{x\in\mathbb{R}^d:\, M_{\lambda,\,T}f(x)>\alpha\}|&\lesssim& \big(1+\log (\frac{1}{\lambda})\big)^{\gamma_1}\int_{\mathbb{R}^d}\frac{|f(x)|}{\alpha}dx\\
&&+\big(1+\log (\frac{1}{\lambda})\big)^{\gamma_2}\int_{\mathbb{R}^d}\Phi\big(\frac{|f(x)|}{\alpha}\big)dx,
\end{eqnarray*}
then for all $p\geq 1$ and $\alpha>0$,
\begin{eqnarray}\label{eqn3.1}|\{x\in\mathbb{R}^d:\, \mathscr{M}_{p,T}f(x)>\alpha\}|\lesssim p^{\gamma_1} \int_{\mathbb{R}^d}\frac{|f(x)|}{\alpha}dx+p^{\gamma_2}\int_{\mathbb{R}^d}\Phi\big(\frac{|f(x)|}{\alpha}\big)dx.\end{eqnarray}
\end{lemma}
\begin{proof}By homogeneity, it suffices to prove (\ref{eqn3.1}) for the case of $\alpha=1$. We will
use the ideas in the proof of Lemma 3.3 in \cite{ler5}. At first, we have that
\begin{eqnarray}\label{eqn3.2}\mathscr{M}_{p,\,T}f(x)\leq \Big(\int^1_0\big(M_{\lambda,\,T}f(x)\big)^pd\lambda\Big)^{1/p}.
\end{eqnarray}
 For $N>0$, denote
$$G_{p,T,N}f(x) =\Big(\int_0^1
\min(M_{\lambda,T}f(x),N)^pd\lambda
\Big)^{1/p}.$$
Set $$\mu_f(t,\,R) = |\{|x|\leq R:\,|f(x)| > t\}| \,\,\hbox{for}\,\, R>0.$$
As in \cite{ler5}, we have that for $k\in\mathbb{N}$,
$$G_{p,T,N}f(x)\leq 2^{-k+1}G_{kp,T,N}f(x)+M_{2^{-kp},T}f(x),
$$
and so
$$\mu_{G_{p,T,N}f} (t,\,R)
\leq
\mu_{G_{kp,T,N}f} (2^{k-2}t,R) +C
(kp)^{\gamma_1}\int_{\mathbb{R}^d}\frac{|f(y)|}{t}dy+C
(kp)^{\gamma_2}\int_{\mathbb{R}^d}\Phi\big(\frac{|f(y)|}{t}\big)dy.
$$
Iterating this estimate $j$-th step yields
\begin{eqnarray*}
\mu_{G_{p,T,N}f} (1,R)
&\leq &
\mu_{G_{k^jp,T,N}f} (2^{(k-2)j},R) +C2^{k-2}\sum_{i=1}^{j}\Big(\frac{
k^{\gamma_1}}{2^{k-2}}\Big)^{i}p^{\gamma_1}\int_{\mathbb{R}^d}|f(y)|dy\\
&&+C2^{k-2}\sum_{i=1}^{j}\Big(\frac{
k^{\gamma_2}}{2^{k-2}}\Big)^{i}p^{\gamma_2}\int_{\mathbb{R}^d}\Phi(|f(y)|)dy.
\end{eqnarray*}
Taking $k=5$ and observing that $\mu_{G_{5^jp,T,N}f} (8^j,R) = 0$, we finally obtain that
$$\mu_{G_{p,T,N}f} (1,\,R)
\lesssim
p^{\gamma_1}\int_{\mathbb{R}^d}|f(y)|dy+p^{\gamma_2}\int_{\mathbb{R}^d}\Phi(|f(y)|)dy.
$$
Letting $N,\,R\rightarrow \infty$ and applying (\ref{eqn3.2}) completes the proof of Lemma \ref{lem3.1}.
\end{proof}

{\it Proof of Theorem \ref{dingli1.3}}. When $\Omega\in L^{\infty}(S^{d-1})$ with $\|\Omega\|_{L^{\infty}(S^{d-1})}=1$, we  have
$$
T^*_\Omega f(x)\leq M  f(x) + T_{\Omega}^{**}f(x),
$$Theorem \ref{dingli2.1}, along with Lemma \ref{lem3.1}, tells us that for $p\in (1,\,\infty)$,
$$ |\{x\in\mathbb{R}^d:\, \mathscr{M}_{p,T_{\Omega}^{**}}f(x)>\alpha\}|\lesssim p \int_{\mathbb{R}^d}\frac{|f(x)|}{\alpha}dx+ \int_{\mathbb{R}^d}\Phi\big(\frac{|f(x)|}{\alpha}\big)dx.$$
On the other hand, it is easy to verify that
$$\mathcal{M}_{M}f(x)\lesssim Mf(x),\,\,\,x\in\mathbb{R}^d.$$
Therefore,  for $p\in (1,\,\infty)$ and $\alpha>0$,
\begin{eqnarray}\label{eq3.3}
|\{x\in\mathbb{R}^d:\, \mathscr{M}_{p,T_{\Omega}^{*}}f(x)>\alpha\}|\lesssim p \int_{\mathbb{R}^d}\frac{|f(x)|}{\alpha}dx+ \int_{\mathbb{R}^d}\Phi\big(\frac{|f(x)|}{\alpha}\big)dx.
\end{eqnarray}
On the other hand, we know that
\begin{eqnarray}\label{equation3.4}
|\{x\in\mathbb{R}^d:\,T_{\Omega}^{*}f(x)>\alpha\}|\lesssim \int_{\mathbb{R}^d}\Phi\big(\frac{|f(x)|}{\alpha}\big)dx,
\end{eqnarray}
see \cite[Theorem 1.1]{bhmo}. These two endpoint estimates,  via the argument used in the proof of Theorem 4.4 in \cite{hulai} (see also the proof of Theorem 3.1 in \cite{ler5}), leads to the conclusion of Theorem \ref{dingli1.3}.
\qed

At the end of this section, we give some applications of Theorem \ref{dingli1.3}.
At first, by mimicking the proof of Theorem 1.1 in \cite{lpr}, we can verify that Theorem \ref{dingli1.3} implies the following weighted inequality of Coifman-Fefferman type.
\begin{theorem}\label{dingli1.5} Let $\Omega$ be homogeneous of degree zero,   have mean value zero on $S^{d-1}$ and $\Omega\in L^{\infty}(S^{d-1})$.
Then for $p\in [1,\,\infty)$ and  $w\in A_{\infty}(\mathbb{R}^d)$,
\begin{eqnarray*}\|T^*_{\Omega}f\|_{L^p(\mathbb{R}^d,\,w)}&\lesssim&\|\Omega\|_{L^{\infty}(S^{d-1})}[w]_{A_{\infty}}^2
\|Mf\|_{L^p(\mathbb{R}^d,\,w)}\\
&&+\|\Omega\|_{L^{\infty}(S^{d-1})}[w]_{A_{\infty}}\log({\rm e}+[w]_{A_{\infty}})
\|M_{\Phi}f\|_{L^p(\mathbb{R}^d,\,w)},
\end{eqnarray*}
provided that $f$ is a bounded function with compact support, where and in the following, $M_{\Phi}$ is the maximal operator defined by
$$M_{\Phi}f(x)=\sup_{Q\ni x}\langle|f |\rangle_{\Phi,\,Q}.$$
\end{theorem}

From (\ref{equa4.1}), we know  that for $r\in (1,\,\infty)$, $$M_{\Phi}f(x)\lesssim M_{r}f(x)+\log (r')Mf(x).$$
This shows that Theorem \ref{dingli1.5} improves   Theorem \ref{dinglidhl} substantially.

To give another application of Theorem \ref{dingli1.3}, we formulate the following weighted weak type endpoint
estimates.
\begin{theorem}\label{thm4.1}
Let $\gamma_1,\,\gamma_2\geq 0$ be two constants, $U$ be a  sublinear operator, $\Psi$ be a Young function such that  $\Psi(t)\geq t$, and for all $p\in (1,\,\infty)$,
$$\Psi(t)\lesssim p'^{\gamma_1}t^{p-1} \,\,\hbox{when}\,\,t\geq 1.
$$
Suppose that  for bounded function  $f$ with compact support, and for each $r\in (1,\,\infty)$, there exists a sparse family $\mathcal{S}$, such that for all bounded function $g$ with compact support,
\begin{eqnarray*}\label{eq4.1}
\Big|\int_{\mathbb{R}^d}g(x)Uf(x)dx\Big|\lesssim r'^{\gamma_2}
\mathcal{A}_{\mathcal{S},\,L^{\Psi},\,L^r}(f,\,g).\end{eqnarray*}
Then for any bounded function $f$ with compact support, $\alpha>0$ and $w\in A_{1}(\mathbb{R}^d)$,
\begin{eqnarray*}\label{extra.4}
w(\{x\in\mathbb{R}^d:|Uf(x)|>\alpha\})
\lesssim [w]^{\gamma_2}_{A_{\infty}}\log^{1+\gamma_1}({\rm e}+[w]_{A_{\infty}})[w]_{A_1}  \int_{\mathbb{R}^d}\Psi\big(\frac{|f(x)|}{\alpha}\big )w(x)dx
.\nonumber
\end{eqnarray*}
\end{theorem}
Theorem \ref{thm4.1} is a generalization of Corollary 3.3 in \cite{hulai}, and the proof is fairly similar to the proof of Corollary 3.3 in \cite{hulai}, we omit the details for brevity.

Obviously,Theorem \ref{dingli1.3}, along with Theorem \ref{thm4.1}, leads to Theorem \ref{dingliweak} directly.

Finally,  we consider the maximal operator associated with the commutator of the operator $T_{\Omega}$ in (\ref{eq1.1}).  Let $b\in {\rm BMO}(\mathbb{R}^d)$. Define the commutator of $T_{\Omega}$ and $b$ by
$$[b,\,T_{\Omega}]f(x)=bT_{\Omega}f(x)-T_{\Omega}(bf)(x).$$
The quantitative weighted estimates for $[b,\,T_{\Omega}]$ were considered by many authors, see \cite{lor,prr,riv, lantaohu} and related references therein. $[b,\,T_{\Omega}]^*$, the maximal operator associated with $[b,\,T_{\Omega}]$, is defined by
$$[b,\,T_{\Omega}]^*f(x)=\sup_{\epsilon>0}\Big|\int_{|x-y|>\epsilon}\frac{\Omega(x-y)}{|x-y|^d}
(b(x)-b(y))f(y)dy\Big|.$$
Using estimate (\ref{eq3.3})-(\ref{equation3.4}), and repeating the proof of Theorem 2.7 in \cite{lantaohu}, we deduce that
\begin{theorem}\label{thm3.4}Let $\Omega$ be homogeneous of degree zero, have mean value zero on $S^{d-1}$, and $\Omega\in L^{\infty}(S^{d-1})$.   Then for each $r\in (1,\,\infty)$ and bounded function $f$, there exists a sparse family $\mathcal{S}$, such that for all bounded function $g$ with compact support,
\begin{eqnarray*}&&\Big|\int_{\mathbb{R}^d}g(x)[b,\,T_{\Omega}]^*f(x)dx\Big|\\
&&\quad\lesssim  \|\Omega\|_{L^{\infty}(S^{d-1})}\big(r'^2\mathcal{A}_{\mathcal{S},\,L^{1},L^r}(f,\,g)+
r'\mathcal{A}_{\mathcal{S},L^{\Phi},L^r}(f,g)\big)\nonumber\\
&&\qquad+\|\Omega\|_{L^{\infty}(S^{d-1})}\big(r'\mathcal{A}_{\mathcal{S},\,L^{\Psi_1},L^r}(f,\,g)+
\mathcal{A}_{\mathcal{S},L^{\Psi_2},L^r}(f,g)\big),\nonumber
\end{eqnarray*}
where and in the following, $\Psi_1(t)=t\log ({\rm e}+t)$, $\Psi_2(t)=t\log ({\rm e}+t)\log\log ({\rm e}^2+t)$.
\end{theorem}
Theorem \ref{thm3.4}, along with Theorem \ref{thm4.1}, leads to the following corollary.
\begin{corollary}\label{cor3.1}
Let $\Omega$ be homogeneous of degree zero, have mean value zero on $S^{d-1}$, and $\Omega\in L^{\infty}(S^{d-1})$. Let $b\in {\rm BMO}(\mathbb{R}^d)$. Then for all $\alpha>0$,
\begin{eqnarray*}
&&w\big(\{x\in\mathbb{R}^d:\,[b,\,T_{\Omega}]^*f(x)>\alpha\}\big)\\
&&\quad\lesssim [w]_{A_1}[w]^2_{A_{\infty}}\log ({\rm e}+[w]_{A_{\infty}})
\int_{\mathbb{R}^d}\Psi_2\big(\frac{|f(x)|}{\alpha}\big)w(x)dx.\nonumber
\end{eqnarray*}
\end{corollary}

As far as we knon, Corollary \ref{cor3.1} is new even for the case $w\equiv 1$.

\bibliographystyle{amsplain}

\end{document}